%% file: main.tex
\documentclass[12pt,a4paper,leqno]{article}

\usepackage{latexsym}
\usepackage{amssymb,exscale}
\usepackage[centertags]{amsmath}
\usepackage{amsthm}

\usepackage[dvips]{hyperref}

\input{preamble}

\begin{document}

\title{From uniform renewal theorem\\
 to uniform large and moderate deviations\\
 for renewal-reward processes}

\author{Boris Tsirelson}

\date{}
\maketitle

\begin{abstract}
A uniform key renewal theorem is deduced from the uniform Blackwell's
renewal theorem. A uniform LDP (large deviations principle) for
renewal-reward processes is obtained, and MDP (moderate deviations
principle) is deduced under conditions much weaker than existence of
exponential moments.
\end{abstract}

\tableofcontents

\section*{Introduction}
\input{intro}

\numberwithin{equation}{section}

\section[Uniform renewal theorems]
  {\raggedright Uniform renewal theorems}
\label{sect1}
\input{sect1}

\section[Uniform large deviations]
  {\raggedright Uniform large deviations}
\label{sect2}
\input{sect2}

\section[Moderate deviations]
  {\raggedright Moderate deviations}
\label{sect3}
\input{sect3}

\bigskip
\filbreak
{
\small
\begin{sc}
\parindent=0pt\baselineskip=12pt
\parbox{4in}{
Boris Tsirelson\\
School of Mathematics\\
Tel Aviv University\\
Tel Aviv 69978, Israel
\smallskip
\par\quad\href{mailto:tsirel@post.tau.ac.il}{\tt
 mailto:tsirel@post.tau.ac.il}
\par\quad\href{http://www.tau.ac.il/~tsirel/}{\tt
 http://www.tau.ac.il/\textasciitilde tsirel/}
}

\end{sc}
}
\filbreak

\end{document}

%% file: preamble.tex
\theoremstyle{definition}
\newtheorem{Theorem}{Theorem}
\swapnumbers
\newtheorem{theorem}[equation]{Theorem}
\newtheorem{lemma}[equation]{Lemma}
\newtheorem{definition}[equation]{Definition}
\newtheorem{remark}[equation]{Remark}

\renewcommand{\phi}{\varphi}

\newcommand{\D}{\mathrm{d}}
\newcommand{\E}{\mathrm{e}}

\newcommand{\ti}{\tilde}

\renewcommand{\(}{\bigl(}
\renewcommand{\)}{\bigr)\vphantom{)}}

\newcommand{\Span}{\operatorname{Span}}

\newcommand{\One}{{1\hskip-2.5pt{\rm l}}}

\newcommand{\eps}{\varepsilon}
\newcommand{\Si}{\Sigma}

\newcommand{\de}{\delta}

\newcommand{\al}{\alpha}
\newcommand{\be}{\beta}

\newcommand{\la}{\lambda}
\newcommand{\La}{\Lambda}

\newcommand{\Ex}{\mathbb E\,}
\newcommand{\R}{\mathbb R}

\renewcommand{\Pr}[1]{\mathbb{P}\mskip1.5mu\(\mskip1.5mu#1\mskip1.5mu\)}

\newcommand{\cE}[2]{\mathbb{E}\mskip1.5mu\(\mskip1.5mu#1\mskip1.5mu
 \big|\mskip1.5mu#2\mskip1.5mu\)}

\newcommand{\close}[1]{$#1$\nobreakdash-\hspace{0pt}close}

%% file: intro.tex
An ordinary renewal-reward process $ S(\cdot) $ is a process of the
form
\[
S(t) = X_1 + \dots + X_n \quad \text{for } \tau_1+\dots+\tau_n \le t
 < \tau_1+\dots+\tau_{n+1} \, ;
\]
here $ (\tau_1,X_1), (\tau_2,X_2), \dots $ are independent copies of a
pair $ (\tau,X) $ of (generally, correlated) random variables such
that $ \tau > 0 $ a.s.

Large deviations principle (LDP) for $ S(t) $ (as $ t\to\infty $) is
well-known when $ \tau $ and $ X $ have exponential moments. Otherwise
the large deviations have peculiarity disclosed
recently \cite{LMZ11}. I prove moderate deviations principle (MDP) for
$ S(t) $ requiring
\begin{gather}
\Ex \tau < \infty \, , \label{*} \\
\Ex \exp(\eps X^2 - \tau) < \infty \quad \text{for some
 } \eps>0 \label{**}
\end{gather}
rather than $ \Ex \exp(\eps |X|) < \infty $, $ \Ex \exp(\eps\tau)
< \infty $. An example: $ X = \sqrt\tau $; MDP holds whenever
$ \Ex \tau < \infty $.

Conditions \eqref{*}, \eqref{**} imply $ \Ex X^2 < \infty $ and are
invariant under linear transformations of $ X $ and rescaling of
$ \tau $ (see Remarks \ref{3.1}, \ref{3.2}), thus, we may restrict
ourselves to the case
\begin{equation}\label{***}
\Ex X = 0 \, , \quad \Ex X^2 = 1 \, , \quad \Ex \tau = 1 \, .
\end{equation}

\begin{Theorem}\label{Th1}
If \eqref{**}, \eqref{***} are satisfied then
\[
\lim_{x\to\infty,x/\sqrt t\to0} \frac1{x^2} \ln \Pr{ S(t) > x\sqrt t }
= -\frac12 \, .
\]
\end{Theorem}

The limit in two variables $t,x$ is taken; that is, for every $ \eps>0
$ there exists $ \de>0 $ such that for all $ t,x $ satisfying $ x >
1/\de $, $ x/\sqrt t<\de $ the function is \close{\eps} to the limit.

Theorem \ref{Th1} (MDP) will be deduced from Theorem \ref{Th3}, and
Theorem \ref{Th3} extends Theorem \ref{Th2} (uniform LDP). The
assumption for Theorem \ref{Th2} is weaker than \eqref{**}:
\begin{equation}\label{****}
\forall \la \in \R \; \forall \eps > 0 \quad \Ex \exp ( \la X
- \eps \tau ) < \infty \, .
\end{equation}
(In combination with \eqref{*} it implies $ \Ex |X| < \infty $, see
Remark \ref{2.25}.)

\begin{Theorem}\label{Th2}
If \eqref{*}, \eqref{****} hold and $ \Ex X = 0 $ then for every $ \la
$, first, there exists one and only one $ \eta_\la \in [0,\infty) $
such that
\begin{equation}\label{*****}
\Ex \exp ( \la X - \eta_\la \tau ) = 1 \, ;
\end{equation}
and second,
\begin{equation}\label{******}
\frac1t \ln \Ex \exp \la S(t) = \eta_\la + O \Big( \frac1t \Big)
\end{equation}
as $ t \to \infty $, uniformly in $ \la \in [-C,-c] \cup [c,C] $
whenever $ 0 < c < C < \infty $.
\end{Theorem}

\begin{Theorem}\label{Th3}
If \eqref{**} and \eqref{***} hold then
\[
\eta_\la = \frac12 \la^2 + o(\la^2) \quad \text{as } \la \to 0 \, ,
\]
and \eqref{******} holds uniformly in $ \la \in [-C,C] $ whenever $ 0 <
C < \infty $.
\end{Theorem}

%% file: sect1.tex
A uniform version of Blackwell's renewal theorem is
available \cite[Th. 1]{WW96}, \cite[Th. 2.6(2), 2.7]{BF??} and may be
formulated as follows.

First, we define the \emph{span} of a probability measure $ \mu $ on $
(0,\infty) $ as
\[
\Span(\mu) = \max \( \big\{ \de>0 : \mu(\{\de,2\de,3\de,\dots\}) =
1 \big\} \cup \big\{0\big\} \) \, ; 
\]
$ \Span(\mu) = 0 $ if and only if $ \mu $ is nonlattice. A set $ M $
of probability measures on $ (0,\infty) $ will be called a \emph{set
of constant span} $ \de $, if $ \Span(\mu) = \de $ for all $ \mu \in M
$. Symbolically: $ \Span(M) = \de $. Thus, a set of constant span $ 0
$ contains only nonlattice measures; a set of constant span $ \de>0 $
contains only lattice measures of span $ \de $ (rather than $ 2\de,
3\de, \dots $).

Second, for every probability measure $ \mu $ on $ (0,\infty) $ we
introduce the \emph{renewal measure} as the sum of convolutions:
\begin{equation}\label{1.05}
U_\mu = \sum_{n=0}^\infty \underbrace{ \mu * \dots * \mu }_n
\end{equation}
(the term for $ n=0 $ being the atom at the origin); $ U_\mu $ is not
finite but locally finite, since $ \int \E^{-t} U_\mu(\D t)
= \sum_n \( \int \E^{-t} \mu(\D t) \)^n < \infty $. It is well-known
(see \cite[p. 123]{BB05}) that $ U_\mu ((u,u+v]) \le U_\mu ([0,v]) $
and moreover,
\begin{equation}\label{!}
U_\mu ((u,u+v]) \le U_\mu ([0,v)) \quad \text{for all } u,v \ge 0 \, .
\end{equation}

\begin{theorem}\label{1.2} (\cite{WW96}, \cite{BF??})
Assume that a set $ M $ of probability measures on $(0,\infty)$ is
weakly compact (treated as a set of measures on $\R$), is a set of
constant span, and is uniformly integrable, that is,
\[
\lim_{a\to+\infty} \sup_{\mu\in M} \int_{[a,\infty)} t \, \mu(\D t) =
0 \, .
\]
Then in the nonlattice case ($ \Span(M)=0 $), for every $ v>0 $,
\[
U_\mu \( [u,u+v) \) \to \frac{ v }{ \int t \, \mu(\D t)
} \quad \text{as } u \to \infty
\]
uniformly in $ \mu \in M $; and in the lattice case ($ \Span(M)=\de
$)
\[
U_\mu (\{n\de\}) \to \frac{ \de }{ \int t \, \mu(\D t)
} \quad \text{as } n \to \infty
\]
uniformly in $ \mu \in M $.
\end{theorem}

The uniform integrability of $ M $ ensures continuity of the function
$ \mu \mapsto \int t \, \mu(\D t) $ on $ M $. We denote for
convenience
\[
\la_\mu = \frac1{ \int t \, \mu(\D t) } \, ;
\]
by compactness,
\begin{equation}\label{1.4}
0 < \min_{\mu\in M} \la_\mu \le \max_{\mu\in M} \la_\mu < \infty \, .
\end{equation}
A uniform version of key renewal theorem follows. We start with the
lattice case.

\begin{theorem}\label{1.45}
Let $ M $ be a set of probability measures on $(0,\infty)$ satisfying
the conditions of Theorem \ref{1.2}, $ \Span(M) = \de > 0 $, and $ H $
a set of functions $ \{0,\de,2\de,\dots\} \to \R $ such that
\[
\sup_{h\in H} \sum_{k=0}^\infty |h(k\de)| < \infty \quad \text{and}
\quad \lim_{n\to\infty} \sup_{h\in H} \sum_{k=n}^\infty |h(k\de)| =
0 \, .
\]
Then
\[
(U_\mu*h) (n\de) \to \de \la_\mu \sum_{k=0}^\infty h(k\de) \quad
\text{as } n \to \infty
\]
uniformly in $ \mu \in M $ and $ h \in H $.
\end{theorem}

\begin{proof}
By \eqref{!}, $ U_\mu(\{n\de\}) \le U_\mu(\{0\}) = 1 $ for all $ \mu $
and $ n $. By Theorem \ref{1.2}, $ U_\mu(\{n\de\}) \to \de \la_\mu $
as $ n \to \infty $, uniformly in $ \mu \in M $. Lemma \ref{1.5}
(below) completes the proof.
\end{proof}

\begin{lemma}\label{1.5}
Let $ U $ and $ H $ be sets of functions $ \{0,1,2,\dots\} \to \R $
such that
\begin{gather*}
\sup_{u\in U} \sup_n |u(n)| < \infty \, , \\
\text{the limit } u(\infty) = \lim_{n\to\infty} u(n) \text{ exists
 uniformly in } u \in U \, ; \\
\sup_{h\in H} \sum_n |h(n)| < \infty \, ; \\
\sum_{n=N}^\infty |h(n)| \to 0 \text{ as } N \to \infty \, , \text{
 uniformly in } h \in H \, .
\end{gather*}
Then
\[
(u*h)(n) \to u(\infty) \sum_{k=0}^\infty h(k) \text{ as }
n \to \infty \, , \text{ uniformly in } u \in U \text{ and } h \in
H \, .
\]
\end{lemma}

\begin{proof}
Denoting $ \|u\|_\infty = \sup_n |u(n)| $, $ \|h\|_1 = \sum_n |h(n)|
$ and $ \Si(h) = \sum_n h(n) $ we have $ \|
u*h \|_\infty \le \|u\|_\infty \|h\|_1 $, $
|u(\infty)| \le \|u\|_\infty $ and $ |\Si(h)| \le \|h\|_1 $. For
arbitrary $ N \in \{0,1,2,\dots\} $ and $ h \in H $ we introduce $
h_N, h^N : \{0,1,2,\dots\} \to \R $ by $ h_N(n) = h(n) $ for $ n \le N
$, $ h_N(n) = 0 $ for $ n > N $, and $ h^N = h - h_N $. We have
$ \sup_{u\in U} \|u\|_\infty < \infty $, $ \sup_{h\in H} \|h\|_1
< \infty $, and $ \sup_{h\in H} \| h^N \|_1 \to 0 $ as $ N \to \infty
$. For arbitrary $ N $ and all $ n \ge N $,
\begin{multline*}
| (u*h)(n) - u(\infty) \Si(h) | \le \\
\le | (u*h_N)(n) - u(\infty) \Si(h_N) | + | (u*h^N)(n) -
 u(\infty) \Si(h^N) | \le \\
\le \Big| \sum_{k=0}^N u(n-k) h(k) - u(\infty) \sum_{k=0}^N h(k) \Big|
 + | (u*h^N)(n) | + | u(\infty) \Si(h^N) | \le \\
\le \sum_{k=0}^N | u(n-k) - u(\infty) | | h(k)| + \| u*h^N \|_\infty +
 |u(\infty)| | \Si(h^N) | \le \\
\le \|h\|_1 \sup_{k\ge n-N} | u(k) - u(\infty) | +
 2 \|u\|_\infty \|h^N\|_1 \, ;
\end{multline*}
given $ \eps > 0 $, we choose $ N $ such that
$ \|u\|_\infty \|h^N\|_1 \le \eps $ for all $ u \in U $ and $ h \in H
$; then for all $ n $ large enough we have $ \|h\|_1 \sup_{k\ge n-N} |
u(k) - u(\infty) | \le \eps $ for all $ u \in U $ and $ h \in H $, and
finally, $ | (u*h)(n) - u(\infty) \Si(h) | \le 3\eps $.
\end{proof}

The nonlattice case needs more effort. Recall that a function $ h :
[0,\infty) \to \R $ is called \emph{directly Riemann integrable,} if
two limits exist and ere equal (and finite):
\[
\lim_{\de\to0+} \de \sum_{n=0}^\infty \inf_{[n\de,n\de+\de)}
h(\cdot)
= \lim_{\de\to0+} \de \sum_{n=0}^\infty \sup_{[n\de,n\de+\de)}
h(\cdot) \, .
\]

\begin{definition}\label{d}
A set $ H $ of functions $ [0,\infty) \to \R $ is \emph{uniformly
directly Riemann integrable,} if
\begin{gather*}
\sup_{h\in H} \sum_{n=0}^\infty \sup_{[n,n+1)} |h(\cdot)| < \infty \,
 , \\
\sum_{n=N}^\infty \sup_{[n,n+1)} |h(\cdot)| \to 0 \quad \text{as
 } N\to\infty \, , \text{ uniformly in } h \in H \, ; \\
\de \sum_{n=0}^\infty \Big( \sup_{[n\de,n\de+\de)} h(\cdot)
 - \inf_{[n\de,n\de+\de)} h(\cdot) \Big) \to 0 \quad \text{as
 } \de\to0+ \, , \text{ uniformly in } h \in H \, .
\end{gather*}
\end{definition}

\begin{remark}\label{1.7}
If $ \sup_{h\in H} \sum_{n=0}^\infty \sup_{[n\de,n\de+\de)} |h(\cdot)|
< \infty $ for some $ \de $ then it holds for all $ \de $. Proof:
given $ \de_1, \de_2 > 0 $, we consider $ A = \{ (n_1,n_2) :
[n_1\de_1,n_1\de_1+\de_1) \cap
[n_2\de_2,n_2\de_2+\de_2) \ne \emptyset \} $, note that $ \#\{ n_1 :
(n_1,n_2) \in A \} \le \frac{\de_2}{\de_1} + 2 $, and get
\begin{multline*}
\sum_{n_1=0}^\infty \sup_{[n_1\de_1,n_1\de_1+\de_1)} |h(\cdot)|
 \le \sum_{n_1=0}^\infty \max_{n_2:(n_1,n_2)\in
 A} \sup_{[n_2\de_2,n_2\de_2+\de_2)} |h(\cdot)| \le \\
\le \sum_{(n_1,n_2)\in A} \sup_{[n_2\de_2,n_2\de_2+\de_2)}
 |h(\cdot)| \le \Big( \frac{\de_2}{\de_1} + 2 \Big)
 \sum_{n_2=0}^\infty \sup_{[n_2\de_2,n_2\de_2+\de_2)} |h(\cdot)| \, .
\end{multline*}
Thus, the first two conditions of Def.~\ref{d} may be reformulated as
\begin{gather*}
\sup_{h\in H} \sum_{n=0}^\infty \sup_{[n\de,n\de+\de)} |h(\cdot)|
 < \infty \, , \\
\sum_{n=N}^\infty \sup_{[n\de,n\de+\de)} |h(\cdot)| \to
 0 \quad \text{as } N\to\infty \, , \text{ uniformly in } h \in H
\end{gather*}
for some (therefore, all) $ \de > 0 $.
\end{remark}

\begin{remark}\label{1.8}
Similarly,
\[
\de \sum_{n:n\de>N} \Big( \sup_{[n\de,n\de+\de)} h(\cdot)
 - \inf_{[n\de,n\de+\de)} h(\cdot) \Big) \le
   (1+2\de) \sum_{n=N}^\infty \sup_{[n,n+1)} |h(\cdot)| \, .
\]
Thus, the third condition of Def.~\ref{d} may be reformulated as
uniform Riemann integrability on bounded intervals: for every $ N $,
\[
\de \sum_{n\ge0:n\de\le N} \Big( \sup_{[n\de,n\de+\de)} h(\cdot)
 - \inf_{[n\de,n\de+\de)} h(\cdot) \Big) \to 0 \quad \text{as
 } \de\to0+ \, , \text{ uniformly in } h \in H \, .
\]
\end{remark}

\begin{remark}\label{1.85}
If each $ h \in H $ is a decreasing function $ [0,\infty) \to
[0,\infty) $ then $ H $ is uniformly directly Riemann integrable if
and only if
\begin{gather*}
\sup_{h\in H} h(0) < \infty \, , \quad \sup_{h\in H} \int_0^\infty
 h(s) \, \D s < \infty \, , \quad \text{and} \\
\sup_{h\in H} \int_a^\infty h(s) \, \D s \to 0 \quad \text{as }
 a \to \infty \, .
\end{gather*}
By taking differences, a similar result can be obtained for functions
of uniformly bounded variation on $ [0,\infty) $ (rather than
decreasing).
\end{remark}

\begin{theorem}\label{t}
Let $ M $ be a set of probability measures on $(0,\infty)$ satisfying
the conditions of Theorem \ref{1.2}, $ \Span(M) = 0 $, and $ H $ a
uniformly directly Riemann integrable set of functions $
[0,\infty) \to \R $. Then
\[
(U_\mu*h) (t) \to \la_\mu \int_0^\infty h(s) \, \D s \quad
\text{as } t \to \infty
\]
uniformly in $ \mu \in M $ and $ h \in H $.
\end{theorem}

Here is a generalization of Lemma \ref{1.5}, to be used in the proof
of the theorem.

\begin{lemma}\label{1.10}
Let $ H $ be as in Lemma \ref{1.5}, and $ V $ a set of functions
$ \{0,1,2,\dots\} \times [0,\infty) \to \R $ such that, first,
$ \sup_{v\in V} \sup_n \sup_t |v_n(t)| < \infty $, and second, the
limit $ v(\infty) = \lim_{t\to\infty} v_n(t) $ exists uniformly in $
v \in V $ for every $ n $, and does not depend on $ n $. Then
\[
\sum_{n=0}^\infty h(n) v_n(t) \to v(\infty) \sum_{n=0}^\infty
h(n) \text{ as } t \to \infty \, , \text{ uniformly in } v \in
V \text{ and } h \in H \, .
\]
\end{lemma}

\begin{proof}
The proof of Lemma \ref{1.5} needs only trivial modifications: $
\sum_n h(n) v_n(t) $ instead of $ (u*h)(n) $; $ \sum_{n=0}^N
|v_n(t)-v(\infty)| |h(n)| $ instead of $ \sum_{k=0}^N
|u(n-k)-u(\infty)| |h(k)| $; and $ \max_{n=0,\dots,N}
|v_n(t)-v(\infty)| $ (for large $ t $) instead of $ \sup_{k\ge
n-N} |u(k)-u(\infty)| $ (for large $ n $). Also, $ \|v\|_\infty
= \sup_{n,t} |v_n(t)| $.
\end{proof}

Here is a special case of Theorem \ref{t} for step functions.

\begin{lemma}\label{1.11}
Assume that $ M $ and $ H $ are as in Theorem \ref{t}, $ \de>0 $, and
every $ h \in H $ is constant on each $ [n\de, n\de+\de) $. Then the
conclusion of Theorem \ref{t} holds.
\end{lemma}

\begin{proof}
Lemma \ref{1.10} will be applied to $ \ti H $ and $ V $, where $ \ti H
$ consists of all $ \ti h $ of the form $ \ti h(n) = h(n\de) $ for $
h \in H $, and $ V $ consists of all $ v $ of the form
\[
v_n(\cdot) = U_\mu * \One_{[n\de,n\de+\de)}
\]
for $ \mu \in M $; that is, $ v_n(t) = U_\mu \( (t-n\de-\de,t-n\de] \)
$. By \eqref{!},
\[
v_n(t) \le U_\mu \( [0,\de) \) \le \E^\de \int \E^{-s} U_\mu(\D s)
= \frac{\E^\de}{ 1 - \int \E^{-s} \mu(\D s) } \, ;
\]
by compactness of $ M $,
\[
\sup_{v,n,t} |v_n(t)| \le \frac{\E^\de}{ 1
- \max_\mu \int \E^{-s} \mu(\D s) } < \infty \, .
\]
By Theorem \ref{1.2}, for every $ n $, $ v_n(t) \to \la_\mu \de $ as $
t \to \infty $, uniformly in $ v $. Thus, $ V $ satisfies the
conditions of Lemma \ref{1.10}. By Remark \ref{1.7}, $ \ti H $
satisfies the conditions (for $ H $) of Lemma \ref{1.10}, that is, of
Lemma \ref{1.5}. It remains to apply Lemma \ref{1.10} and take into
account that $ v(\infty) = \la_\mu \de $, $ \de \sum_n \ti h(n)
= \int_0^\infty h(s) \, \D s $ and $ \sum_n \ti h(n) v_n(\cdot) =
U_\mu * h $ since $ \sum_n h(n\de) \One_{[n\de,n\de+\de)} = h $.
\end{proof}

\begin{proof}[Proof of Theorem \ref{t}]
For arbitrary $ \de > 0 $ and $ h \in H $ we introduce $ h^-_\de,
h^+_\de : [0,\infty) \to \R $ by
\[
h^-_\de(t) = \inf_{[n\de,n\de+\de)} h(\cdot) \, , \quad h^+_\de(t)
= \sup_{[n\de,n\de+\de)} h(\cdot) \quad \text{for } t \in
[n\de,n\de+\de) \, ,
\]
then $ h^-_\de \le h \le h^+_\de $. The sets $ H^-_\de = \{ h^-_\de :
h \in H \} $, $ H^+_\de = \{ h^+_\de : h \in H \} $ are uniformly
directly Riemann integrable by the arguments of Remarks \ref{1.7},
\ref{1.8}. Applying Lemma \ref{1.11} to $ M $ and $ H^\pm_\de $ we get
\[
(U_\mu*h^\pm_\de) (t) \to \la_\mu \int_0^\infty h^\pm_\de(s) \, \D
s \quad \text{as } t \to \infty
\]
uniformly in $ \mu \in M $ and $ h \in H $.

Given $ \eps > 0 $, we choose $ \de = \de_\eps $ such that $ \int |
h^\pm_\de(t) - h(t) | \, \D t \le \eps $ for all $ h \in H $. Then we
choose $ t_\eps $ such that for all $ t \ge t_\eps $, $ \mu \in M $
and $ h \in H $,
\[
\Big| (U_\mu*h^\pm_\de) (t) - \la_\mu \int_0^\infty h^\pm_\de(s) \, \D
s \Big| \le \eps \, .
\]
We get
\begin{multline*}
(U_\mu*h)(t) - \la_\mu \int h(s) \, \D s \le \\
\le (U_\mu*h^+_\de) (t) - \la_\mu \int_0^\infty h^+_\de(s) \, \D s
 + \la_\mu \Big( \int_0^\infty h^+_\de(s) \, \D s - \int_0^\infty
 h(s) \, \D s \Big) \le \\
\le \eps + \la_\mu \eps
\end{multline*}
and a similar lower bound; thus, using \eqref{1.4},
\[
\Big| (U_\mu*h)(t) - \la_\mu \int h(s) \, \D s \Big| \le \eps \Big( 1
  + \max_{\mu\in M} \la_\mu \Big)
\]
for all $ t \ge t_\eps $, $ \mu \in M $ and $ h \in H $.
\end{proof}

%% file: sect2.tex
Theorem \ref{Th2} is proved in this section.

Exponential moments of a renewal-reward process boil down to a renewal
equation, see \cite[Th. 5]{GW94}, and therefore to an auxiliary
renewal process, as sketched below.

Having $ \eta_\la $ satisfying \eqref{*****} for a given $ \la $, we
introduce a random variable $ \tau_\la $ distributed so that
\begin{equation}\label{2.1}
\Ex f(\tau_\la) = \Ex \( \E^{\la X - \eta_\la \tau} f(\tau) \)
\end{equation}
for all bounded Borel functions $ f : \R \to \R $. That is,
\begin{equation}\label{2.13}
\begin{gathered}
\tau \text{ is distributed } \mu, \quad \tau_\la \text{ is distributed
 } \mu_\la, \\
\frac{ \D\mu_\la }{ \D\mu } (\tau) = \cE{ \E^{\la X - \eta_\la \tau}
}{\tau} \, .
\end{gathered}
\end{equation}
Then, using the notation $ \Ex (Z;A) =
\Ex(Z\cdot\One_A) $, we have
\[
\Ex ( \E^{\la S(t)}; \, \tau_1+\dots+\tau_n \le t <
\tau_1+\dots+\tau_{n+1} ) = \E^{\eta_\la t} \Ex h_\la ( t -
\tau_{\la,1} - \dots - \tau_{\la,n} ) \, ,
\]
where $ \tau_{\la,1}, \tau_{\la,2}, \dots $ are independent copies of
$ \tau_\la $, and
\begin{equation}\label{2.17}
h_\la (t) = \begin{cases}
 \E^{-\eta_\la t} \Pr{ \tau_\la>t } &\text{for } t \ge 0, \\
 0 &\text{for } t < 0.
\end{cases}
\end{equation}
Summing up we get
\begin{equation}\label{2.2}
\Ex \E^{\la S(t)} = \E^{\eta_\la t} \sum_{n=0}^\infty \Ex h_\la (
t - \tau_{\la,1} - \dots - \tau_{\la,n} ) = \E^{\eta_\la t} (
U_\la * h_\la ) (t) \, ;
\end{equation}
here $ U_\la = U_{\mu_\la} $ is the renewal measure
(recall \eqref{1.05}).

Recall assumptions \eqref{*} $ \Ex \tau < \infty $ and \eqref{****}
$ \forall \la \in \R \; \forall \eps > 0 \quad \Ex \exp ( \la X
- \eps \tau ) < \infty $.

\begin{remark}\label{2.25}
Assumptions \eqref{*}, \eqref{****} imply $ \Ex |X| < \infty $.
\end{remark}

\begin{proof}
$ |X| \le \tau + \E^{|X|-\tau} \le \tau + \E^{-X-\tau} + \E^{X-\tau} $
is integrable.
\end{proof}

From now on, till the end of this section, we assume the conditions of
Theorem \ref{Th2}; that is, \eqref{*}, \eqref{****}, and $ \Ex X = 0
$. We also assume that $ \Pr{ X=0 } \ne 1 $; otherwise
Theorem \ref{Th2} is trivial.

\begin{lemma}\label{2.3}
Maps $ (\la,\eta) \mapsto \exp ( \la X - \eta \tau ) $ and $
(\la,\eta) \mapsto \tau \exp ( \la X - \eta \tau ) $ are continuous
from $ \R \times (0,\infty) $ to the space $ L_1 $ of integrable
random variables.
\end{lemma}

\begin{proof}
It is sufficient to prove the continuity on $ [-C,C] \times
[2\eps,\infty) $ for arbitrary $ C,\eps>0 $. Also, it is sufficient to
consider the map $ (\la,\eta) \mapsto \E^{\eps\tau} \exp ( \la X
- \eta \tau ) $, since $ \tau \le \frac1{\E\eps} \E^{\eps\tau} $
a.s. We apply the dominated convergence theorem, taking into account
that $ \exp ( -C X - \eps \tau ) + \exp ( C X - \eps \tau ) $ is an
integrable majorant of $ \E^{\eps\tau} \exp ( \la X - \eta \tau ) $
for all $ \la \in [-C,C] $ and $ \eta \in [2\eps,\infty) $.
\end{proof}

\begin{lemma}\label{2.4}
For every $ \la $ there is one and only one $ \eta_\la $
satisfying \eqref{*****} $ \Ex \exp ( \la X - \eta_\la \tau ) = 1 $,
and the function $ \la \mapsto \eta_\la $ is continuous on $ \R $.
\end{lemma}

\begin{proof}
The function $ \psi : \R \times (0,\infty) \to (0,\infty) $ defined by
$ \psi(\la,\eta) = \Ex \exp ( \la X - \eta \tau ) $ is continuous by
Lemma \ref{2.3}. For every $ \la $ the function $ \psi(\la,\cdot) $ is
strictly decreasing, $ \psi(\la,+\infty) = 0 $, and (possibly,
infinite) $ \psi(\la,0+) = \Ex \exp \la X > \exp \la \Ex X = 1 $
provided that $ \la \ne 0 $. Thus, for $ \la \ne 0 $ we get unique
$ \eta_\la > 0 $; and trivially, $ \eta_0 = 0 $.

It remains to prove continuity of the function $ \la \mapsto \eta_\la
$. Given $ \la_0 \ne 0 $ and $ \eps < \eta_{\la_0} $ we note that
$ \psi(\la_0,\eta_{\la_0}+\eps) < 1 = \psi(\la_0,\eta_{\la_0})
< \psi(\la_0,\eta_{\la_0}-\eps) $ and take $ \de > 0 $ such that
$ \psi(\la,\eta_{\la_0}+\eps) < 1 = \psi(\la,\eta_\la)
< \psi(\la,\eta_{\la_0}-\eps) $ and therefore $ \eta_{\la_0}-\eps
< \eta_\la < \eta_{\la_0}+\eps $ for all $ \la \in
(\la_0-\de,\la_0+\de) $. For $ \la_0 = 0 $ we use a one-sided version
of the same argument: given $ \eps > 0 $, we take $ \de > 0 $ such
that $ \psi(\la,\eps) < 1 $ and therefore $ \eta_\la < \eps $ for all
$ \la \in (-\de,\de) $.
\end{proof}

Recall measures $ \mu, \mu_\la $ given by \eqref{2.13}.

\begin{lemma}\label{2.5}
The function $ \la \mapsto \mu_\la $ is continuous from $
(-\infty,0) \cup (0,\infty) $ to the space of measures with the norm
topology.
\end{lemma}

\begin{proof}
We have $ \D\mu_\la / \D\mu = \phi_\la $, where $ \phi_\la $ is
defined by $ \phi_\la(\tau) = \cE{ \E^{\la X - \eta_\la \tau } }{\tau}
$. If $ \la_n \to \la \ne 0 $ then, using Lemmas \ref{2.3}
and \ref{2.4}, $ \| \mu_{\la_n} - \mu_\la \| = \int | \phi_{\la_n}
- \phi_\la | \, \D\mu = \Ex | \phi_{\la_n}(\tau) - \phi_\la(\tau) |
= \Ex | \cE{ \E^{\la_n X - \eta_{\la_n} \tau } }{\tau} - \cE{ \E^{\la
X - \eta_\la \tau } }{\tau} | \le \Ex | \E^{\la_n X
- \eta_{\la_n} \tau } - \E^{\la X - \eta_\la \tau } | \to 0 $ as $
n \to \infty $.
\end{proof}

\begin{lemma}\label{2.6}
The set $ \{ \mu_\la : \la \in [-C,-c] \cup [c,C] \} $ satisfies the
conditions of Theorem \ref{1.2} whenever $ 0 < c < C < \infty $.
\end{lemma}

\begin{proof}
Lemma \ref{2.5} ensures compactness (even in a topology stronger than
needed). For every $ \la $ measures $ \mu_\la $ and $ \mu $ are
mutually absolutely continuous, therefore $ \Span(\mu_\la)
= \Span(\mu) $. It remains to prove uniform
integrability. Using \eqref{2.1} we have
$ \int_{[a,\infty)} t \, \mu_\la(\D t)
= \Ex \( \tau_\la \One_{[a,\infty)}(\tau_\la) \)
= \Ex \( \tau \exp(\la X-\eta_\la \tau) \One_{[a,\infty)}(\tau) \) \to
0 $ as $ a\to\infty $ uniformly in $ \la \in [-C,-c] \cup [c,C] $,
since random variables $ \tau \exp(\la X-\eta_\la \tau) $ for these
$ \la $ are a compact subset of $ L_1 $ by Lemmas \ref{2.3}, \ref{2.4}.
\end{proof}

Recall functions $ h_\la $ given by \eqref{2.17}.

\begin{lemma}\label{2.7}
The set $ \{ h_\la : \la \in [-C,-c] \cup [c,C] \} $ is uniformly
directly Riemann integrable whenever $ 0 < c < C < \infty $.
\end{lemma}

\begin{proof}
Follows from Remark \ref{1.85}, \eqref{1.4} and the uniform
integrability of measures $ \mu_\la $, since $ h_\la(0) \le 1 $ and
\begin{multline*}
\int_a^\infty h_\la(t) \, \D t = \int_a^\infty \E^{-\eta_\la
 t} \Pr{ \tau_\la > t } \, \D t \le \int_a^\infty \mu_\la \(
 (t,\infty) \) \, \D t = \\
= \int_a^\infty (t-a) \, \mu_\la(\D t) \le \int_{[a,\infty)}
 t \, \mu_\la(\D t) \, .
\end{multline*}
\end{proof}

\begin{lemma}\label{2.8}
The map $ \la \mapsto h_\la $ is continuous from $ (-\infty,0) \cup
(0,\infty) $ to $ L_1(0,\infty) $.
\end{lemma}

\begin{proof}
For every $ t > 0 $, $ h_\la(t) = \E^{-\eta_\la t} \mu_\la \(
(t,\infty) \) $ is continuous in $ \la \ne 0 $ by
Lemmas \ref{2.4}, \ref{2.5}. Also, $ h_\la(t) \le 1 $. Thus,
$ \la \mapsto h_\la |_{(0,a)} \in L_1(0,a) $ is continuous (by the
dominated convergence theorem). The limit as $ a \to \infty $ is
locally uniform around a given $ \la \ne 0 $ due to the uniform
integrability of $ h_\la $ (proved in Lemma \ref{2.7}).
\end{proof}

Functions $ h_\la $ as elements of $ L_1(0,\infty) $ are relevant in
the nonlattice case, when $ \Span(\mu)=0 $, while in the lattice case,
when $ \Span(\mu) = \de > 0 $, we treat sequences $ \(
h_\la(k\de) \)_{k=0}^\infty $ as elements of the space $ l_1 $ of
summable sequences.

\begin{lemma}\label{2.9}
Let $ \Span(\mu) = \de > 0 $, then the map $ \la \mapsto \(
h_\la(k\de) \)_{k=0}^\infty $ is continuous from $ (-\infty,0) \cup
(0,\infty) $ to $ l_1 $.
\end{lemma}

\begin{proof}
We have
\[
h_\la(k\de) = \frac1\de \int_{k\de}^{k\de+\de} \E^{\eta_\la(t-k\de)}
h_\la(t) \, \D t \, ,
\]
since the function $ t \mapsto \E^{\eta_\la t} h_\la(t) = \mu_\la\(
(t,\infty) \) $ is constant on $ [k\de,k\de+\de) $. We apply
Lemma \ref{2.8}, taking into account continuity of
$ \la \mapsto \eta_\la $.
\end{proof}

\begin{proof}[Proof of Theorem \ref{Th2}]
Existence and uniqueness of $ \eta_\la $ satisfying \eqref{*****} are
ensured by Lemma \ref{2.4}.

We reformulate \eqref{******} as existence of $ T \in (0,\infty) $
such that
\begin{equation}\label{2.10}
\sup_{\la \in [-C,-c] \cup [c,C], \, t \in [T,\infty)} \Big| -\eta_\la
t + \ln \Ex \E^{\la S(t)} \Big| < \infty \, .
\end{equation}

The set $ M = \{ \mu_\la : \la \in [-C,-c] \cup [c,C] \} $ satisfies
the conditions of Theorem \ref{1.2} by Lemma \ref{2.6}.

By \ref{1.4}, $ \int t \, \mu_\la(\D t) $ is bounded away from $ 0 $
and $ \infty $ for $ \la \in [-C,-c] \cup [c,C] $.
The rest of the proof of \eqref{2.10} splits in two cases.

Nonlattice case: $ \Span(\mu)=0 $.

The set $ H = \{ h_\la : \la \in [-C,-c] \cup [c,C] \} $ is uniformly
directly Riemann integrable by Lemma \ref{2.7}. By \eqref{2.2} and
Theorem \ref{t},
\[
\E^{-\eta_\la t} \Ex \E^{\la S(t)} \to \frac{ \int_0^\infty
h_\la(s) \, \D s }{ \int s \, \mu_\la(\D s) } \quad \text{as }
t \to \infty
\]
uniformly in $ \la \in [-C,-c] \cup [c,C] $. In order to
get \eqref{2.10} it remains to check that the right-hand side is
bounded away from $ 0 $ and $ \infty $. For the denominator, see
above. For the numerator, use continuity of the function
$ \la \mapsto \int_0^\infty h(t) \, \D t $ for $ \la\ne0 $
(Lemma \ref{2.8}).

Lattice case: $ \Span(\mu) = \de > 0 $.

The set $ H $ of restrictions to $ \{ 0, \de, 2\de, \dots \} $ of the
functions $ h_\la $ for $ \la \in [-C,-c] \cup [c,C] $ satisfies the
conditions of Theorem \ref{1.45} by Lemma \ref{2.9} (via compactness in
$ l_1 $). By \eqref{2.2} and Theorem \ref{1.45},
\[
\E^{-\eta_\la n\de} \Ex \E^{\la
S(n\de)} \to \frac{ \de \sum_{k=0}^\infty h_\la(k\de) }{ \int
s \, \mu_\la(\D s) } \quad \text{as } n \to \infty 
\]
uniformly in $ \la \in [-C,-c] \cup [c,C] $.
The function $ \la \mapsto \sum_{k=0}^\infty h_\la(k\de) $ is continuous
for $ \la\ne0 $ by Lemma \ref{2.9}, therefore the sum is bounded away
from $ 0 $ and $ \infty $ for $ \la \in [-C,-c] \cup [c,C] $ (it
cannot vanish since $ h_\la(0)=1 $), which leads to \eqref{2.10}.
Thus we get \eqref{2.10} for $ t $ running on the lattice, which is
enough, since $ S(\cdot) $ is constant on $ [k\de,k\de+\de) $ (and
$ \eta_\la $ is bounded).
\end{proof}

%% file: sect3.tex
Theorems \ref{Th1} and \ref{Th3} are proved in this section.

In order to use small $ \la $ we need \eqref{**}: $ \Ex \exp ( \eps
X^2 - \tau ) < \infty $ for some $ \eps > 0 $.

\begin{remark}\label{3.1}
Assumptions \eqref{*}, \eqref{**} imply $ \Ex X^2 < \infty $.
\end{remark}

\begin{proof}
$ \eps X^2 \le \tau + \E^{\eps X^2 - \tau} $ is integrable.
\end{proof}

\begin{remark}\label{3.2}
Assumption \eqref{**} is invariant under linear transformations of $ X
$, and rescaling of $ \tau $; also, \eqref{**} implies \eqref{****}.
\end{remark}

\begin{proof}
Rescaling $ X $: $ \Ex \exp \( (c^{-2}\eps) (cX)^2 - \tau \)
= \Ex \exp ( \eps X^2 - \tau ) < \infty $.

Shifting $ X $: $ \Ex \exp \( \frac\eps2 (X+c)^2
- \tau \) \le \Ex \exp \( \frac\eps2 (X-c)^2 + \frac\eps2 (X+c)^2
- \tau \) = \E^{c^2 \eps} \Ex \exp ( \eps X^2 - \tau ) < \infty $.

Rescaling $ \tau $: $ \Ex \exp ( c \eps X^2 - c \tau ) = \Ex \( \exp
( \eps X^2 - \tau ) \)^c \le \( \Ex \exp ( \eps X^2 - \tau ) \)^c
< \infty $ for $ c \in (0,1) $, and $ \Ex \exp ( \eps X^2 - c \tau
) \le \Ex \exp ( \eps X^2 - \tau ) < \infty $ for $ c \in [1,\infty)
$.

Finally, \eqref{**} implies \eqref{****} since $ \Ex \exp(\de
X^2-\tau) < \infty $ implies $ \Ex \exp(\eps\de X^2-\eps\tau) < \infty
$ (assuming $ 0<\eps<1$) and therefore $ \Ex \exp(\la
X-\eps\tau) \le \Ex \exp \( \frac{\la^2}{4\eps\de} + \eps\de X^2
- \eps\tau \) < \infty $.
\end{proof}

From now on, till the end of this section, we assume the conditions of
Theorem \ref{Th3}; that is, \eqref{**}, and \eqref{***}: $ \Ex X = 0
$, $ \Ex X^2 = 1 $, $ \Ex \tau = 1 $. Conditions of Theorem \ref{Th2}
follow, since \eqref{**} implies \eqref{****} by Remark \ref{3.2}.

Here is an analytic fact that will give us some integrable majorants.

\begin{lemma}\label{3.3}
For all $ a, \eps, \La \in (0,\infty) $,
\[
\sup_{t>0,x>0,\la\in(0,\La)} \frac{ (1+t+x^2) \exp( \la x - a \la^2 t
) }{ t + \exp( \eps x^2 - t ) } < \infty \, .
\]
\end{lemma}

\begin{proof}
Denoting this supremum by $ S(a,\eps,\La) $ we observe that $
S(a,\eps,\La) \le \max(1,c^2) S(c^{-2}a, c^2 \eps, c \La ) $ for
arbitrary $ c > 0 $ (by rescaling, $ x \mapsto c x $ and $ \la \mapsto
c^{-1} \la $). Thus, we restrict ourselves to $ \eps = 1 $.

We note that
\[
\max_{\la\in\R} (\la x-a\la^2 t) = \frac{ x^2 }{ 4at } \, .
\]
We choose $ \al,\be>0 $ such that $ \al > 1 $ and $ \be^2 <
4a(\al^2-1) $ (for instance, $ \al=2 $ and $ \be=3\sqrt a $) and
consider three cases.

Case 1: $ x \le \al \sqrt t $.

We note that $ t + \exp(x^2-t) \ge t + \E^{-t} \ge \max(t,1) $, thus,
\begin{multline*}
\frac{ (1+t+x^2) \exp( \la x - a \la^2 t ) }{ t + \exp( x^2 - t ) } \le 
 \frac{ (1+t+x^2) \exp \frac{ x^2 }{ 4at } }{ \max(t,1) } \le \\
\le \frac{ (1+t+\al^2 t) \exp \frac{ \al^2 }{ 4a } }{ \max(t,1) } \le
 (2+\al^2) \exp \frac{ \al^2 }{ 4a } \, .
\end{multline*}

Case 2: $ \al \sqrt t \le x \le \be t $.
\begin{multline*}
\frac{ (1+t+x^2) \exp( \la x - a \la^2 t ) }{ t + \exp( x^2 - t ) } \le 
 \frac{ (1+t+x^2) \exp \frac{ x^2 }{ 4at } }{ \exp(\al^2 t - t) } \le \\
\le (1+t+\be^2 t^2) \exp \Big( \frac{ \be^2 t^2}{ 4at } - \al^2 t +
 t \Big) \le \\
\le \sup_{t>0} (1+t+\be^2 t^2) \exp \Big( \! - \frac{ 4a(\al^2-1)-\be^2
}{ 4a } t \Big) < \infty \, .
\end{multline*}

Case 3: $ x \ge \be t $.
\begin{multline*}
\frac{ (1+t+x^2) \exp( \la x - a \la^2 t ) }{ t + \exp( x^2 - t )
 } \le \\
\le ( 1 + \be^{-1} x + x^2 ) \exp ( \la x - a \la^2 t - x^2 + t ) \le \\
\le \sup_x ( 1 + \be^{-1} x + x^2 ) \exp ( \La x - x^2 + \be^{-1} x )
< \infty \, .
\end{multline*}
\end{proof}

\begin{lemma}\label{3.4}
For all $ a, \eps, \La \in (0,\infty) $,
\[
\sup_{t>0,x\in\R,\la\in(-\La,\La)} \frac{ (1+t+x^2) \( 1 + \exp( \la x -
a \la^2 t ) \) }{ t + \exp( \eps x^2 - t ) } < \infty \, .
\]
\end{lemma}

\begin{proof}
By Lemma \ref{3.3} applied to $ |x|, |\la| $,
\[
\sup_{t>0,x\in\R,\la\in(-\La,\La)} \frac{ (1+t+x^2) \exp( |\la x| -
a \la^2 t ) }{ t + \exp( \eps x^2 - t ) } < \infty \, ,
\]
and $ \la x \le |\la x| $, of course. The new terms are bounded:
\begin{gather*}
\frac{ 1+t }{ t + \exp( \eps x^2 - t ) } \le \frac{ 1+t }{ t + \E^{-t}
 } \le \frac{ 1+t }{ \max(1,t) } \le 2 \, ; \\
\frac{ x^2 }{ t + \exp( \eps x^2 - t ) } \le \frac{ x^2 }{ t + \eps
x^2 - t } \le \frac1\eps \, .
\end{gather*}
\end{proof}

Here is a counterpart of Lemma \ref{2.3}. This time, the origin
$ \la=\eta=0 $ is included (but its neighborhood is reduced).

\begin{lemma}\label{3.45}
For every $ a \in (0,\infty) $, maps $ (\la,\eta) \mapsto \exp ( \la X
- \eta \tau ) $ and $ (\la,\eta) \mapsto \tau \exp ( \la X - \eta \tau
) $ are continuous from $ \{ (\la,\eta) : \la \in \R, \eta \in
[a\la^2,\infty) \} $ to the space $ L_1 $ of integrable random
variables.
\end{lemma}

\begin{proof}
We apply the dominated convergence theorem, taking into account that
$ \tau + \exp ( \eps X^2 - \tau ) $ is an integrable majorant by
Lemma \ref{3.4}.
\end{proof}

\begin{lemma}\label{3.5}
For all $ a, \eps, \La \in (0,\infty) $,
\[
\sup_{t>0,x\in\R,\la\in(-\La,0)\cup(0,\La)} \frac{ | \exp( \la x -
a \la^2 t ) - 1 - ( \la x - a \la^2 t ) | }{ \la^2 \( t + \exp( \eps
x^2 - t ) \) } < \infty \, .
\]
\end{lemma}

\begin{proof}
Denote $ u = \la x - a \la^2 t $.

Case $ |x| \le a |\la| t $:
we have $ |\la x| \le a \la^2 t $, thus $ -2a \la^2 t \le u \le 0 $
and $ | \E^u - 1 - u | = \E^u - 1 - u \le 1 - 1 - u = -u \le 2 a \la^2
t \le 2 a \la^2 \( t + \exp( \eps x^2 - t ) \) $.

Case $ |x| \ge a |\la| t $:
we apply the bound $ | \E^u - 1 - u | \le \frac12 u^2 \max(1,\E^u) $,
note that $ u^2/\la^2 \le 2x^2 + 2(a\la t)^2 \le 4x^2 $ and get an
upper bound
\[
\frac{ x^2 \max \( 1, \exp( \la x - a \la^2 t ) \) }{ t + \exp( \eps
x^2 - t ) } \, ,
\]
bounded by Lemma \ref{3.4}.
\end{proof}

\begin{lemma}\label{3.6}
For all $ a \in (0,\infty) $,
\[
\frac{ \Ex \exp( \la X - a \la^2 \tau ) - 1 }{ \la^2 } \to \frac12 -
a \quad \text{as } \la \to 0 \, .
\]
\end{lemma}

\begin{proof}
We have
\[
\frac{ \exp( \la X - a \la^2 \tau ) - 1 - ( \la X - a \la^2 \tau )
}{ \la^2 } \to \frac12 X^2 \quad \text{a.s.} \quad \text{as } \la \to
0 \, .
\]
The left-hand side is dominated by $ \tau + \exp(\eps X^2 - \tau) $ by
Lemma \ref{3.5}, the majorant being integrable (for some $ \eps $)
by \eqref{*}, \eqref{**}. By the dominated convergence theorem,
\[
\frac{ \Ex \exp( \la X - a \la^2 \tau ) - 1 - \la \Ex X +
a \la^2 \Ex \tau }{ \la^2 } \to \frac12 \Ex X^2 \, ;
\]
it remains to use \eqref{***}.
\end{proof}

Recall $ \eta_\la $ satisfying \eqref{*****} $ \Ex \exp ( \la X
- \eta_\la \tau ) = 1 $, given by Lemma \ref{2.4}; $ \eta_0 = 0 $, and
$ \eta_\la > 0 $ for $ \la \ne 0 $.

\begin{lemma}\label{3.7}
$ \eta_\la = \frac12 \la^2 + o(\la^2) $ as $ \la \to 0 $.
\end{lemma}

\begin{proof}
If $ a > \frac12 $ then by Lemma \ref{3.6}, $ \Ex \exp ( \la X -
a \la^2 \tau ) < 1 $ and therefore $ \eta_\la < a \la^2 $ for all
$ \la \ne 0 $ small enough. Similarly, if $ a < \frac12 $ then
$ \eta_\la > a \la^2 $ for all $ \la \ne 0 $ small enough.
\end{proof}

\begin{lemma}\label{3.8}
The function $ \la \mapsto \mu_\la $ is continuous from $ \R $ to the
space of measures with the norm topology.
\end{lemma}

\begin{proof}
Continuity on $ (-\infty,0) \cup (0,\infty) $ holds by
Lemma \ref{2.5}. The same proof gives now continuity at $ 0 $ due to
Lemma \ref{3.45} (and \ref{3.7}).
\end{proof}

\begin{lemma}\label{3.10}
The set $ \{ \mu_\la : \la \in [-C,C] \} $ satisfies the conditions of
Theorem \ref{1.2} whenever $ 0 < C < \infty $.
\end{lemma}

\begin{proof}
We repeat the proof of Lemma \ref{2.6} using Lemma \ref{3.8} instead
of \ref{2.5}, and \ref{3.45} instead of \ref{2.3}.
\end{proof}

Recall the functions $ h_\la(t) = \E^{-\eta_\la
t} \mu_\la\((t,\infty)\) $. Similarly to Lemma \ref{2.7} we get
uniform direct Riemann integrability of the set $ \{ h_\la : \la \in
[-C,C] \} $. Similarly to Lemma \ref{2.8}, the map $ \la \mapsto h_\la
$ is continuous from $ \R $ to $ L_1(0,\infty) $. Similarly to
Lemma \ref{2.9}, in the nonlattice case the map $ \la \mapsto \(
h_\la(k\de) \)_{k=0}^\infty $ is continuous from $ \R $ to $ l_1 $.

\begin{proof}[Proof of Theorem \ref{Th3}]
The first claim is given by Lemma \ref{3.7}. For the second claim, the
proof of Theorem \ref{Th2} needs only trivial modifications: $ [-C,C]
$ and related results of Sect.~\ref{sect3} are used instead of $
[-C,-c] \cup [c,C] $ and related results of Sect.~\ref{sect2}.
\end{proof}

\begin{proof}[Proof of Theorem \ref{Th1}]
By Theorem \ref{Th3},
\[
\frac1{\la^2 t} \ln \Ex \exp \la S(t) = \frac{ \eta_\la }{ \la^2 } +
O \Big( \frac1{ \la^2 t } \Big) = \frac12 + o(1) +
O \Big( \frac1{ \la^2 t } \Big)
\]
as $ t \to \infty $, uniformly in $ \la \in [-C,0)\cup(0,C] $. Thus,
\[
\lim_{t\to\infty, \la\to0, \la^2
t\to\infty} \frac1{\la^2 t} \ln \Ex \exp \la S(t) = \frac12 \, ;
\]
Theorem \ref{Th1} follows by the well-known G\"artner(-Ellis)
argument.
\end{proof}